\documentclass[a4paper,reqno,11pt,oneside]{amsart}
\usepackage[all]{xy}
\usepackage{km}

\title{Spherical functors on the Kummer surface}
\author{Andreas Krug and Ciaran Meachan}
\address{Mathematisches Institut, Universit\"{a}t Bonn, Deutschland}
\email{akrug@math.uni-bonn.de}
\address{School of Mathematics, University of Edinburgh, Scotland}
\email{ciaran.meachan@ed.ac.uk}

\begin{document}

\begin{abstract}
We find two natural spherical functors associated to the Kummer surface and analyse how their induced twists fit with Bridgeland's conjecture on the derived autoequivalence group of a complex algebraic K3 surface.
\end{abstract}

\maketitle
{\let\thefootnote\relax\footnotetext{2010 \emph{Mathematics Subject Classification}: 14F05, 14J28, 18E30.

A.K. was supported by the SFB/TR 45 `Periods, Moduli Spaces and Arithmetic of Algebraic Varieties' of the DFG (German Research Foundation) and C.M. was supported by an EPSRC Doctoral Prize Research Fellowship Grant no. EP/K503034/1.}}

\section{Introduction}
Let $\cD(X)$ be the bounded derived category of coherent sheaves on a smooth complex projective variety $X$ and $\Aut(\cD(X))$ denote the set of isomorphism classes of exact $\bbC$-linear autoequivalences of $\cD(X)$. Then we always have a subgroup $\Aut_\trm{st}(\cD(X))\subset \Aut(\cD(X))$ of \emph{standard} autoequivalences which is generated by push forwards along automorphisms, twists by line bundles and shifts. The complement of this subgroup, if non-empty, is usually very interesting and mysterious; its elements will be called \emph{non-standard} autoequivalences.

The most successful way to construct non-standard autoequivalences was discovered in the groundbreaking work of Seidel and Thomas \cite{seidel2001braid} on \emph{spherical objects}. This was extended by Huybrechts and Thomas \cite{huybrechts2006pobjects} to a notion of \emph{$\bbP$-objects} and further still, to a theory of \emph{spherical} and \emph{$\bbP$-functors}; see \cite{rouquier2006categorification,anno2008weak,addington2011new}.

The first example of a series of $\bbP$-functors was constructed by Addington in \cite[Theorem 2]{addington2011new} for the Hilbert scheme $X^{[n]}$ of $n$ points on a K3 surface $X$. In particular, he showed that the natural functor $F:\cD(X)\to\cD(X^{[n]})$ induced by the universal ideal sheaf on $X\times X^{[n]}$ is a $\bbP^{n-1}$-functor in the sense of \cite[\S 3]{addington2011new} and thus gives rise to a non-standard autoequivalence of $\cD(X^{[n]})$ for each $n\geq2$. Notice that when $n=1$, this $F$ is Mukai's reflection functor \cite[p.362]{mukai1987moduli} which coincides (up to a shift) with the spherical twist around the structure sheaf $\cO_X$. 

Inspired by this example, the second author \cite[Theorem 4.1]{meachan2012derived} provided the analogous result for the generalised Kummer variety $K_n\subset A^{[n+1]}$ associated to an abelian surface $A$. 
More precisely, he proved that the natural Fourier-Mukai functor $F_K: \cD(A)\to \cD(K_n)$ induced by the universal ideal sheaf on $A\times K_n$ is again a $\bbP^{n-1}$-functor yielding a new non-standard autoequivalence of $\cD(K_n)$ for each $n\ge 2$. 

This short note completes this theorem to the case $n=1$ where the generalised Kummer variety is the classical Kummer surface. The motivation to understand this particular case comes from Bridgeland's conjecture \cite[Conjecture 1.2]{bridgeland2008stability} on the derived autoequivalence group of a complex algebraic K3 surface; 
roughly speaking, it says that $\Aut(\cD(X))$ should be generated by standard autoequivalences and twists around spherical objects.

\subsection*{Summary of main results} Every abelian surface $A$ has a natural K3 surface associated to it; namely the \emph{Kummer surface} $K:=K_1$. It can either be defined as the blow up of the quotient $A/\iota$ along the sixteen ordinary double points, where $\iota$ denotes the involution $a\mapsto-a$, or equivalently as the fibre of the Albanese map $m:A^{[2]}\to A$ over zero. That is, we can identify $K$ with the subvariety of the Hilbert scheme $A^{[2]}$ consisting of those points representing length 2 subschemes of $A$ whose weighted support sums to zero. In other words, there is a universal family $\cZ\subset A\times K$ giving rise to the commutative diagram
\[\xymatrix{&\cZ\ar[dr]^-q\ar[dl]_-p&\\ A\ar[dr]_-\pi && K\ar[dl]^-\mu\\ &A/\iota&}\] 

Recall that a Fourier-Mukai functor
$F:\cD(Y)\to\cD(X)$ with left adjoint $L$ and right adjoint $R$ is said to be \emph{spherical} if the cotwist $C_F:=\cone(\id\xra\eta RF)$ is an autoequivalence of $\cD(Y)$ and we have a functorial isomorphism $R\simeq CL$. In particular, if $F$ is spherical then the \emph{twist} $T_F:=\cone(FR\xra\epsilon\id)$ is an autoequivalence of $\cD(X)$. A spherical object $\cE\in\cD(X)$ corresponds to the case $F:=(\underline{\;\;})\otimes \cE:\cD(\pt)\to\cD(X)$. 

In this article, we focus on the exact triangle $F\to F'\to F''$ of Fourier-Mukai functors $\Phi_\cE:\cD(A)\to \cD(K)$ induced by the structure sequence of $\cZ$:
\[F:=\Phi_{\cI_\cZ}\qquad F':=\Phi_{\cO_{A\times K}}=\H^*(\_)\otimes \cO_K\qquad F'':=\Phi_{\cO_{\cZ}}=q_*p^*.\]

Our main result is the following

\begin{thm*}[\ref{F''spherical} and \ref{Fspherical}]
Both $F$ and $F''$ are spherical functors with cotwists $C_F\simeq C_{F''}\simeq\iota^*$. 
\end{thm*}

In light of \cite[Conjecture 1.2]{bridgeland2008stability}, this immediately raises the question whether the twists $T_F,T_{F''}\in\Aut(\cD(K))$ associated to these \emph{functors} $F,F''$ can be decomposed into twists $T_\cE$ around spherical \emph{objects} $\cE\in\cD(K)$. We answer this question with the following
%

\begin{thm*}[\ref{F''spherical} and \ref{Fspherical}]
The induced twists $T_{F},T_{F''}\in\Aut(\cD(K))$ decompose in the following way:
\[T_{F''}\simeq\prod_iT_{\cO_{E_i}(-1)}^{-1}\circ M_{\cO_K(E/2)}[1]\simeq \prod_iT_{\cO_{E_i}}\circ M_{\cO_K(-E/2)}[1]\]
and 
\[F[1]\simeq T_{\cO_K}\circ F'' \quad\implies\quad T_F\simeq T_{\cO_K}\circ T_{F''}\circ T_{\cO_K}^{-1}\]
where $E=\bigcup_i E_i$ for the exceptional curves $E_i$ of the Hilbert-Chow morphism $\mu$ and $M_{\cO_K(E/2)}:=(\_)\otimes\cO_K(E/2)$.
\end{thm*}

It is easy to see that the squares $T_F^2,T_{F''}^2$ of our twists act trivially on the cohomology of $K$ (see \cite[\S 1.4]{addington2011new}). 
In fact, Corollary \ref{squares} shows that 
$T_F^2\simeq T_{F''}^2\simeq[2]$.
\vs

In this paper, we will give a different proof of Theorem \ref{Fspherical} to that which could have been obtained from adapting the arguments in \cite{meachan2012derived}. The advantage of our approach is that it immediately provides us with the decompositions of $T_F$ and $T_{F''}$ as stated above. 
\vs

\textbf{Acknowledgements}: We thank Nick Addington and Will Donovan for helpful discussions as well as the Hausdorff Research Institute for Mathematics (HIM) for their excellent hospitality whilst this work was carried out. C.M. is very grateful to Arend Bayer for his consistent help and support. 

\section{Natural Functors on the Kummer Surface}
Another way of describing $K$ is by first blowing-up the fixed points $\tilde A\to A$. Since the fixed points are $\iota$-invariant, the involution $\iota$ lifts to an involution $\tilde{\iota}$ of $\tilde{A}$. 
 \[\xymatrix{&\tilde{A}\ar[dr]^{q}\ar[dl]_p&\\ A\ar[dr]_\pi && K
 \ar[dl]^\mu\\ &A/\iota&}\]
The quotient $\tilde A\to K$ is a double cover ramified over sixteen exceptional curves $E_i$. 
Moreover, the canonical bundle formula for the blow-up yields $\omega_{\tilde{A}} \simeq \cO(\sum \tilde{E}_i)$ where the $\tilde{E}_i$ are the exceptional divisors in $\tilde{A}$. Their images $E_i$ in $K$ satisfy $q^*\cO(E_i) \simeq \cO(2\tilde{E}_i)$ and $q_*\cO_{\tilde{A}} \simeq \cO_K \oplus \cO(-\frac{1}{2}\sum E_i)$. See \cite[Chapter 1.1]{huybrechts2014lectures} for more details. We set $E:=\bigcup_i E_i$ and $\tilde{E}:=\bigcup_i \tilde{E}_i$ from now on. 

%

\begin{prop}\label{F''spherical}
$F'':\cD(A)\to\cD(K)$ is a spherical functor with cotwist $C_{F''}\simeq\iota^*$ and twist
\[T_{F''}\simeq\prod_iT_{\cO_{E_i}(-1)}^{-1}\circ M_{\cO_K(E/2)}[1].\]
\end{prop}

\begin{proof}
Pushforward along the double cover $q_*:\cD(\tilde{A})\to\cD(K)$ is a spherical functor with cotwist $C_{q_*}\simeq M_{\cO_{\tilde{A}}(\tilde{E})}\circ\tilde{\iota}^*\simeq S_{\tilde{A}}\circ\tilde{\iota}^*[-2]$ and twist $T_{q_*}\simeq M_{\cO_K(E/2)}[1]$; see \cite[\S 1.2, Examples 5 \& 6]{addington2011new}. 

By \cite[Theorem 4.3]{orlov1993projective}, we have a semi-orthogonal decomposition 
\[\cD(\tilde{A})\simeq\langle \cO_{\tilde{E}_1}(-1),\ldots,\cO_{\tilde{E}_{16}}(-1),p^*\cD(A)\rangle\]
We set $\cA:=\langle \cO_{\tilde{E}_1}(-1),\ldots,\cO_{\tilde{E}_{16}}(-1)\rangle$ and $\cB:=p^*\cD(A)$ so that $\cD(\tilde{A})\simeq\langle\cA,\cB\rangle$. 
Since $\cD(\tilde{A})\simeq\langle S_{\tilde{A}}\cB,\cA\rangle$ by \cite{bondal1989representable} and $C_{q_*}\cB\simeq S_{\tilde{A}}\cB$, we have $\cD(\tilde{A})\simeq\langle C_{q_*}\cB,\cA\rangle$. Thus, by \cite[Theorem 4.13]{halpernleistner2013autoequivalences}, the restrictions $q_*|_{\cA}:\cD(A[2])\to\cD(K)$ (to the set $A[2]\subset A$ of 2-torsion points) and $q_*|_{\cB}\simeq q_*p^*=:F'':\cD(A)\to\cD(K)$ are spherical functors with $T_{q_*}\simeq T_{q_*|_{\cA}}\circ T_{q_*|_{\cB}}$. Since $q_*\cO_{\tilde{E}_i}(-1)\simeq\cO_{E_i}(-1)$, we see that $T_{q_*|_{\cA}}\simeq\prod_i T_{\cO_{E_i}(-1)}$ and hence
\[T_{F''}\simeq T_{q_*|_{\cA}}^{-1}\circ T_{q_*}\simeq\prod_iT_{\cO_{E_i}(-1)}^{-1}\circ M_{\cO_K(E/2)}[1].\]
Notice that the cotwist of $F''\simeq q_*|_{\cB}$ is given by $S_A\circ\iota^*[-2]\simeq \iota^*$. 
\end{proof}

\begin{rmk}
We can use equation \eqref{twistj} below to rewrite this decomposition as \[T_{F''}\simeq \prod_iT_{\cO_{E_i}}\circ M_{\cO_K(-E/2)}[1].\]
\end{rmk}

\begin{lem}\label{F[1]=T_OF''}
We have the following isomorphism of functors \[F[1]\simeq T_{\cO_K} \circ F''.\]
\end{lem}

\begin{proof}
Consider the following exact triangles of functors
\[\Hom^*(\cO_K,F'')\otimes\cO_K\to F'' \to T_{\cO_K}\circ F''\quad\trm{and}\quad F'\to F''\to F[1].\] Then it is sufficient to show that $\Hom^*(\cO_K,F'')\otimes\cO_K\simeq F'\simeq \H^*(A,\_)\otimes \cO_K$. In other words, it is enough to show that $\H^*(K,F''(\_))\simeq\H^*(A,\_)$ but this follows from the fact that $p$ is a blowup. Indeed, we have \[\H^*(K,F''(\_))\simeq\H^*(K,q_*p^*(\_))\simeq\H^*(\tilde{A},p^*(\_))\simeq \H^*(A,p_*p^*(\_))\simeq\H^*(A,\_).\qedhere\]
\end{proof}

\begin{cor}\label{Fspherical}
$F:\cD(A)\to\cD(K)$ is a spherical functor with cotwist $C_F\simeq\iota^*$ and twist
\[T_{F}\simeq T_{\cO_K}\circ T_{F''}\circ T_{\cO_K}^{-1}.\]
\end{cor}

\begin{proof}
Recall that if $F:\cD(Z)\to\cD(Y)$ is a spherical functor and $\Phi:\cD(Y)\xra\sim\cD(X)$ is an equivalence of categories then $\Phi\circ F:\cD(Z)\to\cD(X)$ is also a spherical functor with the same cotwist and $T_{\Phi\circ F}\simeq \Phi\circ T_F\circ \Phi^{-1}$. 
In particular, we see immediately from Lemma \ref{F[1]=T_OF''} that $F$ is a spherical functor with cotwist $C_F\simeq\iota^*$ and twist 
\[T_F\simeq T_{F[1]}\simeq T_{\cO_K}\circ T_{F''}\circ T_{\cO_K}^{-1}.\qedhere\]
\end{proof}

\begin{cor}\label{squares}
The squares of the spherical twists are given by \[T_{F}^2\simeq T_{F''}^2\simeq[2].\] In particular, $T_F^2,T_{F''}^2$ act trivially on cohomology.
\end{cor}

\begin{proof}
Let $j:E\to K$ denote the inclusion of the exceptional divisor. Since $E$ is smooth, we can apply \cite[\S 1.2, Example 5]{addington2011new} to see that $j_*:\cD(E)\to\cD(K)$ is spherical with cotwist $C_{j_*}\simeq M_{\cO_{E}(E)}[-1]\simeq S_E[-2]$ and twist $T_{j_*}\simeq M_{\cO_K(E)}$. 

Set $\cA_1:=\langle\cO_{E_1}(-1),\ldots,\cO_{E_{16}}(-1)\rangle$ and $\cA_2:=\cA_1\otimes\cO_E(1)$ to be subcategories of $\cD(E)$.
Then, by \cite[Theorem 2.6]{orlov1993projective}, we have a semi-orthogonal decomposition
\[\cD(E)\simeq\langle \cA_1,\cA_2\rangle\]
Thus, using Kuznetsov's trick \cite[Theorem 11]{addington2013categories} (which is a special case of \cite[Theorem 4.13]{halpernleistner2013autoequivalences}), we see that the restriction $j_\ell:=j_*|_{\cA_\ell}:\cD(A[2])\to\cD(K)$ is spherical 
for each $\ell=1,2$ and the twists satisfy $T_{j_1}\circ T_{j_2}\simeq T_{j_*}$. That is
\begin{equation}\label{twistj}
\prod_i T_{\cO_{E_i}(-1)}\circ\prod_i T_{\cO_{E_i}}\simeq M_{\cO_K(E)}.
\end{equation}

Furthermore, we have $j_1\simeq M_{\cO_K(E/2)}\circ j_2$ since $\cO_{E_i}(E/2)\simeq\cO_{E_i}(-1)$ and so 
\[T_{j_1}\simeq T_{M_{\cO_K(E/2)}\circ j_2}\simeq M_{\cO_K(E/2)}\circ T_{j_2}\circ M_{\cO_K(-E/2)}\] 
which, after taking inverses, equates to 
\begin{equation}\label{switch}
\prod_iT_{\cO_{E_i}(-1)}^{-1}\circ M_{\cO_K(E/2)}\simeq M_{\cO_K(E/2)}\circ \prod_iT_{\cO_{E_i}}^{-1}.
\end{equation}
This expression allows us to reduce the formula for $T_{F''}^2$ in the following way:
\begin{align*}
T_{F''}^2 &\simeq \prod_iT_{\cO_{E_i}(-1)}^{-1}\circ M_{\cO_K(E/2)}\circ\prod_iT_{\cO_{E_i}(-1)}^{-1}\circ M_{\cO_K(E/2)}[2]\\
&\simeq M_{\cO_K(E/2)}\circ \prod_iT_{\cO_{E_i}}^{-1}\circ\prod_iT_{\cO_{E_i}(-1)}^{-1}\circ M_{\cO_K(E/2)}[2]\\
&\simeq M_{\cO_K(E/2)}\circ M_{\cO_K(-E)}\circ M_{\cO_K(E/2)}[2]\\
&\simeq[2]
\end{align*}
where the second and third lines follow from equations \eqref{switch} and \eqref{twistj} respectively.

The fact that $T_F^2\simeq[2]$ now follows immediately from Corollary \ref{Fspherical}.
\end{proof}

\begin{cor}\label{imFspans}
$\im F$ and $\im F''$ are spanning classes for $\cD(K)$.
\end{cor}

\begin{proof}
For any spherical functor $F:\cD(Y)\to\cD(X)$, we have a natural spanning class for $\cD(X)$ given by $\im F\cup(\im F)^\perp\simeq\im F\cup\ker R$; see \cite[\S 1.4]{addington2011new}. However, in our case we have $\ker R=0$. Indeed, let $\cE\in \ker R$. Then the defining triangle for the twist $FR(\cE)\to\cE\to T_F(\cE)$ shows that $T_F(\cE)\simeq\cE$. But by Corollary \ref{squares}, we have $\cE\simeq T_{F}^2(\cE)\simeq\cE[2]$ which implies $\cE\simeq0$; a similar argument works for $F''$.
\end{proof}

\begin{rmk}
This should be contrasted to the object case where every spherical object $\cE$ is expected to have a non-empty perpendicular $\cE^\perp$; \cite[Question 1.25]{ploog2005groups}.
\end{rmk}

\begin{lem}\label{split}
The functors $F,F'':\cD(A)\to\cD(K)$ are actually split spherical. That is, the natural triangles 
associated to the units $\eta,\eta''$ of adjunction are split. In particular, this implies that $F$ and $F''$ are faithful.
\end{lem}

\begin{proof}
We prove the statement only for $F$ since $F''$ is identical. In order to show that the triangle $\id_A\xra\eta RF\to\iota^*$ is split, it suffices to show that $\Ext^1(\id_A,\iota^*)=0$. But on the level of kernels, this is just 
\begin{align*}
\Ext^1_{A\times A}(\Delta_*\cO_A,\cO_{\Gamma_\iota})&\simeq \Ext^1_{A}(\cO_A,\Delta^!\cO_{\Gamma_\iota})\quad\trm{by adjunction}\\
&\simeq \Ext^1_{A}(\cO_A,\Delta^*\cO_{\Gamma_\iota}[-2])\\
&\simeq \H^{-1}(A,\cO_{A[2]})=0.\qedhere
\end{align*}
\end{proof}

\begin{prop}\label{cohFM}
The induced map on cohomology $F^\H:\H^*(A,\bbQ)\to\H^*(K,\bbQ)$ is injective on $\H^\trm{even}(A,\bbQ)$, zero on $\H^\trm{odd}(A,\bbQ)$ and the twist $T_F$ acts on $\H^*(K,\bbQ)$ by reflection in $(\im F^\H)^\perp$ 
with respect to the Mukai pairing.
\end{prop}

\begin{proof}
The first statement follows from 
the fact that
$R^\H F^\H\simeq\id_{\H^*(A,\bbQ)}+\iota^{*\H}$ and $\iota^{*\H}$ acts by the identity on $\H^\trm{even}(A,\bbQ)$ and by $-1$ on $\H^\trm{odd}(A,\bbQ)$. 
Next, the defining triangle for the twist gives $T_F^\H\simeq \id_{\H^*(K,\bbQ)}-F^\H R^\H$ from which it follows immediately that everything in $\ker R^\H\simeq(\im F^\H)^\perp$ is fixed by $T_F^\H$. Finally, to see that $T_F^\H$ acts on $\im F^\H$ as $-1$ we observe that $T_F\circ F\simeq F\circ C_F[1]\simeq F\circ\iota^*[1]\simeq F[1]$ 
and so the claim follows.
\end{proof}

\begin{rmk}
Notice that this is very different to the object case where the twist acts on cohomology by reflection in a \emph{hyperplane}; see \cite[Corollary 8.13]{huybrechts2006fourier} for more details. It follows from Proposition \ref{cohFM} that our twist is acting on cohomology by reflection in a subspace of codimension $8=\dim\H^{\trm{even}}(A,\bbQ)$.
\end{rmk}

\bibliographystyle{alpha}
\bibliography{ref}

\begin{thebibliography}{Muk87}

\bibitem[AA13]{addington2013categories}
Nicolas Addington and Paul Aspinwall.
\newblock Categories of massless {D}-branes and del {P}ezzo surfaces.
\newblock {\em Arxiv preprint arXiv:1305.5767v1}, 2013.

\bibitem[Add11]{addington2011new}
Nicolas Addington.
\newblock New derived symmetries of some hyperk{\"a}hler varieties.
\newblock {\em Arxiv preprint arXiv:1112.0487v1}, 2011.

\bibitem[Ann08]{anno2008weak}
Irina Anno.
\newblock {\em Weak representation of tangle categories in algebraic geometry}.
\newblock PhD thesis, Harvard University, 2008.

\bibitem[BK89]{bondal1989representable}
Alexei Bondal and Mikhail Kapranov.
\newblock Representable functors, {S}erre functors, and reconstructions.
\newblock {\em Izv. Akad. Nauk SSSR Ser. Mat.}, 53(6):1183--1205, 1337, 1989.

\bibitem[Bri08]{bridgeland2008stability}
Tom Bridgeland.
\newblock Stability conditions on {$K3$} surfaces.
\newblock {\em Duke Math. J.}, 141(2):241--291, 2008.

\bibitem[HLS13]{halpernleistner2013autoequivalences}
Daniel Halpern-Leistner and Ian Shipman.
\newblock Autoequivalences of derived categories via geometric invariant
  theory.
\newblock {\em Arxiv preprint arXiv:1303.5531v1}, 2013.

\bibitem[HT06]{huybrechts2006pobjects}
Daniel Huybrechts and Richard Thomas.
\newblock {$\mathbb P$}-objects and autoequivalences of derived categories.
\newblock {\em Math. Res. Lett.}, 13(1):87--98, 2006.

\bibitem[Huy06]{huybrechts2006fourier}
Daniel Huybrechts.
\newblock {\em Fourier-{M}ukai transforms in algebraic geometry}.
\newblock Oxford Mathematical Monographs. The Clarendon Press Oxford University
  Press, Oxford, 2006.

\bibitem[Huy14]{huybrechts2014lectures}
Daniel Huybrechts.
\newblock {\em Lectures on K3 surfaces}.
\newblock Available online at
  \url{http://www.math.uni-bonn.de/people/huybrech/K3Global.pdf}, 2014.

\bibitem[Mea12]{meachan2012derived}
Ciaran Meachan.
\newblock Derived autoequivalences of generalised {K}ummer varieties.
\newblock {\em Arxiv preprint arXiv:1212.5286v3}, 2012.

\bibitem[Muk87]{mukai1987moduli}
Shigeru Mukai.
\newblock On the moduli space of bundles on {K}3 surfaces. {I}.
\newblock In {\em Vector bundles on algebraic varieties ({B}ombay, 1984)},
  volume~11 of {\em Tata Inst. Fund. Res. Stud. Math.}, pages 341--413. Tata
  Inst. Fund. Res., Bombay, 1987.

\bibitem[Orl92]{orlov1993projective}
Dmitri Orlov.
\newblock Projective bundles, monoidal transformations, and derived categories
  of coherent sheaves.
\newblock {\em Izv. Ross. Akad. Nauk Ser. Mat.}, 56(4):852--862, 1992.

\bibitem[Plo05]{ploog2005groups}
David Ploog.
\newblock {\em Groups of autoequivalences of derived categories of smooth
  projective varieties}.
\newblock PhD thesis, Universit{\"a}t Bonn, 2005.

\bibitem[Rou06]{rouquier2006categorification}
Rapha{\"e}l Rouquier.
\newblock Categorification of {${\mathfrak{sl}}_2$} and braid groups.
\newblock In {\em Trends in representation theory of algebras and related
  topics}, volume 406 of {\em Contemp. Math.}, pages 137--167. Amer. Math.
  Soc., Providence, RI, 2006.

\bibitem[ST01]{seidel2001braid}
Paul Seidel and Richard Thomas.
\newblock Braid group actions on derived categories of coherent sheaves.
\newblock {\em Duke Math. J.}, 108(1):37--108, 2001.

\end{thebibliography}
\end{document}